\documentclass[12pt]{extarticle}
\usepackage[top=70pt,bottom=70pt,left=75pt,right=77pt]{geometry}

\usepackage{amssymb, amsmath, amsthm, setspace, bbm, prettyref, graphicx, float, subfigure, color, mathrsfs, mathtools,comment}

\usepackage{Lutsko_Style}
\DeclareMathSizes{10}{10}{7}{5}

\title{Exceptional eigenvalue density for thin groups}
\author{Christopher Lutsko}

\begin{document}

  \maketitle
  \begin{abstract}
    \noindent  We prove a limit multiplicity conjecture of Hee Oh for principal congruence covers of geometrically finite hyperbolic manifolds, with an explicit power-saving rate. The first estimate follows by combining the pre-trace argument from Oh's work with injectivity on a fixed compact core. We formulate the argument as a shadow-return principle, giving a geometric interpretation to the exceptional eigenvalue density. For convex-cocompact groups, a packing argument for enlarged shadows gives a stronger rate.
  \end{abstract}

  \section{Introduction}

  Let $G:= \SO(n,1)^o$, with maximal compact subgroup $K<G$. Let $\rho=\frac{n-1}{2}$, and let $\Gamma < G(\Z)$ be a torsion-free, geometrically finite, Zariski dense subgroup. Write $X = \Gamma \bk \bH^n$ and let $\delta_{\Gamma}$ denote the critical exponent for $\Gamma$; we assume $\delta_{\Gamma}>\rho$. Elstrodt, Patterson, and Sullivan \cite{E1,E2,P,S} showed that $\delta_{\Gamma}$ is equal to the Hausdorff dimension of the limit set and moreover, satisfies
  \begin{align*}
    \lambda_0(\Gamma) = \delta_\Gamma(n-1-\delta_\Gamma)
  \end{align*}
  where $\lambda_0(\Gamma)$ is the bottom of the $L^2(X)$-spectrum of the Laplacian (equal to $0$ for lattices and greater than or equal to $0$ for non-lattice subgroups). If $X$ is noncompact, the essential spectrum of the Laplacian starts at $\rho^2$; if $X$ is compact, its spectrum is discrete. A major topic of research is the presence of possible \emph{exceptional eigenvalues} below $\rho^2$ (see for example Selberg's eigenvalue conjecture \cite{IK}).

  A natural question\footnote{We thank Dennis Sullivan for posing this question to the author} is whether the exceptional eigenvalues have any geometric interpretation. To which the answer is yes. In particular, we show that the density is controlled by the return of Patterson-Sullivan shadows through a compact core. For convex-cocompact groups the Patterson-Sullivan measure is comparable to the Hausdorff measure. Hence, this allows us to answer Sullivan's question.

  For $s \in (\rho,n-1]$, put
  $
      \lambda_s = s(n-1-s).
  $
      If $\Gamma'<\Gamma$ has finite index, let $\cN(\lambda_s,\Gamma')$ denote the number of Laplace eigenvalues of $L^2(\Gamma'\bk \bH^n)$ in $[0,\lambda_s]$, counted with multiplicity. For $q \ge 1$, set
      \begin{align*}
        \Gamma(q) : = \{\gamma\in\Gamma \ : \ \gamma \equiv e \mod{q}\}, \qquad D_q : = [\Gamma: \Gamma(q)].
      \end{align*}
      In \cite[Conjecture 1.11]{H}, Hee Oh conjectured that the normalized exceptional eigenvalue count tends to zero along the principal congruence tower. We prove this conjecture with polynomial decay. 

      \begin{theorem}[Oh's Limit Multiplicity Conjecture] \label{thm:main}
        For every $s \in (\rho, \delta_\Gamma]$ and every $q\ge 3$,
          \begin{align}\label{eq:main}
            \frac{\cN(\lambda_s,\Gamma(q))}{D_q} \ll_s q^{-(s-\rho)}.
          \end{align}
        If $\Gamma$ is convex cocompact, then, for every $\epsilon>0$,
          \begin{align}\label{eq:main-cc}
            \frac{\cN(\lambda_s,\Gamma(q))}{D_q}
            \ll_{s,\epsilon}
            q^{-\frac{\delta_\Gamma(s-\rho)}{2(\delta_\Gamma-\rho)}+\epsilon}.
          \end{align}
      \end{theorem}
      The first estimate in Theorem~\ref{thm:main} is a short consequence of the pre-trace argument in \cite[Section~2.3]{H}. Indeed, the positive kernel, the collar lemma, the spectral lower bound, and the unfolding over the compact collar are already present there. The new input for \eqref{eq:main} is the estimate in Lemma~\ref{lem:injectivity}: stopping Oh's kernel before the first nonidentity return leaves only the identity term on the geometric side. In particular, this deduction does not require the uniform spectral gap or the global orbital-counting theorem used in \cite{H}. The shadow-return formulation below abstracts this observation, while the second estimate \eqref{eq:main-cc} uses the additional packing argument of Section~\ref{sec:congruence}.

      Note that, for all but finitely many primes, strong approximation gives $D_q \asymp q^g$ where $g = \dim (G)$ \cite{MVW}. Hence, Theorem 1 gives
      \begin{align}\label{eq:degree}
        \cN(\lambda_s,\Gamma(q)) \ll_s D_q^{1-(s-\rho)/g}.
      \end{align}
      In the convex-cocompact case, \eqref{eq:main-cc} similarly gives
      \begin{align}\label{eq:degree-cc}
        \cN(\lambda_s,\Gamma(q))
        \ll_{s,\epsilon}
        D_q^{1-\frac{\delta_\Gamma(s-\rho)}{2g(\delta_\Gamma-\rho)}+\epsilon}.
      \end{align}
      This improves \eqref{eq:degree} whenever $\delta_\Gamma<2\rho$.

      \subsection{Geometry and spectral density}

      We now state the geometric result behind Theorem \ref{thm:main}. Let
      \begin{align*}
        p_i : X_i \to X
      \end{align*}
      be finite covers of degree $D_i$. Fix $s \in (\rho,\delta_\Gamma]$. We say that $(X_i)$ has the \emph{uniform collar property at $s$} if there exists a compact subset $\Omega=\Omega_s \subset X$ and a constant $c_s>0$ such that every $L^2$-normalized eigenfunction $\phi$ on $X_i$ with spectral parameter $u\in[s,\delta_\Gamma]$ satisfies
        \begin{align}\label{CP}
          \int_{p_i^{-1}(\Omega)} |\phi|^2 \rd x \ge c_s.
        \end{align}
        This prevents exceptional eigenfunctions from escaping into the ends. By the collar lemma of Gamburd and Magee \cite{G,M}, principal congruence covers have the uniform collar property at every $s\in(\rho,\delta_\Gamma]$.

        Put $\Omega_i=p_i^{-1}(\Omega)$. For $x =\Gamma_i g K\in \Omega_i$, set $r_g(\gamma) = d(go,\gamma go)$ and define the normalized return energy to be
        \begin{align}\label{eq:energy}
          \cE_i(L) := \frac{1}{D_i} \int_{\Omega_i} \sum_{\substack{\gamma \in \Gamma_i \\ r_g(\gamma) \le L }} e^{-\rho r_g(\gamma)} \rd x,
         \end{align}
        where the sum is independent of the representative $g$.

        \begin{theorem}[Shadow-return transfer]\label{thm:transfer}
          Suppose $(X_i)$ has the uniform collar property at $s$. Then, for every $L \ge 2$,
          \begin{align}\label{eq:transfer}
            \frac{\cN(\lambda_s,\Gamma_i)}{D_i} \ll_s e^{-(s-\rho)L} \cE_i(L).
          \end{align}
          The implied constant is independent of $i$ and $L$.
        \end{theorem}

        \begin{remark}
          We call this a shadow return lemma for the following reason. Given $z,w\in \bH^n$ and $R>0$, the \emph{shadow} of $B_w(R)$ viewed from $z$ is defined to be
          \begin{align*}
            \operatorname{Sh}_z(B_w(R)) : = \{\xi \in \partial \bH^n \ : \ [z, \xi) \cap B_w(R) \neq \varnothing\},
          \end{align*}
          where $[z,\xi)$ denotes the geodesic ray connecting $z$ to $\xi$. Thus, the shadow consists of all boundary points whose geodesic ray from $z$ intersects $B_w(R)$. 
            Let $(\nu_x)$ denote the Patterson-Sullivan density (see below) and fix $R_0$ sufficiently large. Then the shadow lemma \cite{S} gives, uniformly over lifted compact cores,
        \begin{align}\label{eq:shadow}
          e^{-\rho r_g(\gamma)} \asymp \nu_{g o}(\operatorname{Sh}_{go}(B_{\gamma g o}(R_0)))^{\rho/\delta_{\Gamma}}.
        \end{align}
        Hence, $\cE_i(L)$ is an averaged return energy for the Patterson-Sullivan shadows.
        \end{remark}

      \subsection{Previous work}

      \cite[Conjecture 1.11]{H} is known to be true if $\Gamma$ is a cocompact lattice by DeGeorge and Wallach \cite{DW}, or if $\Gamma=\SL_2(\Z)$ by Sarnak \cite{Sa}. It does not seem to be known for a general arithmetic lattice, although Savin proved it for those $s$ whose corresponding eigenfunctions are cusp forms \cite{Sav}. We also refer to \cite{FLM}, where an analogous problem was answered positively for principal congruence subgroups of $\GL_n$ and $\SL_n$.

      The quantitative density theorems of Sarnak--Xue and Oh give the following comparison.

      \begin{theorem}[Sarnak--Xue \cite{SX}; Oh \cite{H}]\label{thm:SX-Oh}
        \leavevmode
        \begin{enumerate}
          \item If $\Gamma$ is a cocompact arithmetic subgroup of $\SO(n,1)^o$, with $n=2$ or $3$, then, for every fixed $s\in(\rho,n-1]$ and every $\epsilon>0$,
          \begin{align}\label{eq:SX-density}
            \cN(\lambda_s,\Gamma(q))
            \ll_{s,\epsilon}D_q^{((n-1)-s)/\rho+\epsilon}.
          \end{align}
          \item Under the hypotheses of this paper, there exists $\eta>0$ such that, for every fixed $s\in(\rho,\delta_\Gamma]$ and every prime $q$,
          \begin{align}\label{eq:Oh-density}
            \cN(\lambda_s,\Gamma(q))
            \ll_s D_q^{(\delta_\Gamma-s)/\eta}.
          \end{align}
          Here $\eta$ may be any positive number smaller than the uniform spherical spectral gap of the family.
        \end{enumerate}
      \end{theorem}

      The estimate \eqref{eq:Oh-density} is stronger than \eqref{eq:main} when $s$ is sufficiently close to $\delta_\Gamma$, but it gives limit multiplicity only when $\delta_\Gamma-s<\eta$. The estimate \eqref{eq:main} covers the full exceptional interval without a uniform spectral gap. Its proof uses the same positive kernel, collar lemma, pre-trace inequality, and unfolding as \cite[Section~2.3]{H}; Lemma~\ref{lem:injectivity} permits the kernel to be stopped before the global orbital-counting estimate is needed. The common pre-trace deduction of the earlier comparison estimates is recalled in Section~\ref{sec:transfer}.

      The classical limit-multiplicity theorem of DeGeorge and Wallach normalizes by total volume \cite{DW}, while later work treats Benjamini--Schramm convergence \cite{ABB}. In the infinite-volume setting, the collar property prevents escape of mass and the degree of one fixed compact core replaces total volume. The return condition is closely related to the weak injective-radius property of Golubev--Kamber \cite{GK}. The additional geometric point developed here is that the same pre-trace argument can be organized in terms of a boundary-geometric quantity, namely the return energy of Patterson--Sullivan shadows. This formulation also permits nonidentity returns to be retained and controlled, leading to the convex-cocompact improvement \eqref{eq:main-cc}.

      \subsection*{Notation}

      The identity element of $G$ is denoted by $e$, and $o=eK\in\bH^n$. Haar measure on $G$ is normalized so that the induced measure on $G/K$ is hyperbolic volume. We write $A\ll_\vartheta B$ if $|A|\leq C_\vartheta B$ for a constant depending only on the indicated parameters and the fixed base manifold. The notation $B_w(R)$ denotes the hyperbolic ball of radius $R$ centered at $w$.

      \section{Preliminaries}\label{sec:preliminaries}

      \subsection{Patterson--Sullivan densities}

      For $\xi\in\partial\bH^n$, let
      \begin{align*}
        \beta_\xi(z,w)
        =\lim_{y\to\xi}\bigl(d(z,y)-d(w,y)\bigr)
      \end{align*}
      denote the Busemann cocycle. A Patterson--Sullivan density of dimension $\delta_\Gamma$ is a family $(\nu_z)_{z\in\bH^n}$ of finite measures supported on the limit set $\Lambda(\Gamma)$ satisfying
      \begin{align}\label{eq:PS-density}
        \gamma_*\nu_z=\nu_{\gamma z},
        \qquad
        \frac{\rd\nu_z}{\rd\nu_w}(\xi)
        =e^{-\delta_\Gamma\beta_\xi(z,w)}.
      \end{align}
      We fix one such density. For a geometrically finite, Zariski dense group it is unique up to scale. The shadow lemma gives
      \begin{align}\label{eq:shadow-lemma}
        \nu_z\!\left(\operatorname{Sh}_z(B_{\gamma z}(R_0))\right)
        \asymp e^{-\delta_\Gamma d(z,\gamma z)}
      \end{align}
      for $R_0$ sufficiently large, uniformly when the image of $z$ ranges over a fixed compact subset of $X$ \cite{P,S}. If $\Gamma$ is convex cocompact, then $\nu_z$ is Ahlfors $\delta_\Gamma$-regular for a visual metric based at $z$, and hence comparable to the $\delta_\Gamma$-dimensional Hausdorff measure on $\Lambda(\Gamma)$. Raising \eqref{eq:shadow-lemma} to the power $\rho/\delta_\Gamma$ gives \eqref{eq:shadow}.

      \subsection{Exceptional spectrum and positive kernels}

      We take the Laplacian to be nonnegative. Its spectrum below $\rho^2$ is discrete, with finite multiplicities, and every eigenvalue in this range has a unique parameter $u\in(\rho,2\rho]$ such that
      \begin{align*}
        \lambda_u=u(2\rho-u)=u(n-1-u).
      \end{align*}
      Since $u\mapsto\lambda_u$ is decreasing on $[\rho,2\rho]$, $\cN(\lambda_s,\Gamma')$ counts precisely the exceptional parameters $u\geq s$, with multiplicity.

      Let $\psi_u$ be the positive-definite spherical function corresponding to the parameter $u$, normalized by $\psi_u(e)=1$. If $f\in L^1(K\bk G/K)$, its spherical transform is
      \begin{align}\label{eq:spherical-transform}
        \widehat f(\lambda_u)=\int_G f(g)\psi_u(g)\,\rd g.
      \end{align}
      We write $\check f(g)=\overline{f(g^{-1})}$. Then $f*\check f$ is of positive type and has spherical transform $|\widehat f|^2\geq0$. For every compact interval $J\subset(\rho,2\rho]$, one has, uniformly for $u\in J$,
      \begin{align}\label{eq:spherical-decay}
        \psi_u(a_t)\asymp_J e^{(u-2\rho)t},
      \end{align}
      where $d(o,a_to)=t$.

      For each $q$, let $s_{1,q}<\delta_\Gamma$ be the largest spherical complementary-series parameter below $\delta_\Gamma$ occurring in $L^2(\Gamma(q)\bk\bH^n)$, with $s_{1,q}=\rho$ if there is no such parameter. The family has a \emph{uniform spherical spectral gap} if
      \begin{align*}
        \liminf_q(\delta_\Gamma-s_{1,q})>0.
      \end{align*}

      \subsection{Congruence covers and compact cores}

      Fix the defining integral representation $G\subset\GL_{n+1}$ and put $G(\Z)=G\cap\GL_{n+1}(\Z)$. A subgroup of $G$ is arithmetic if it is commensurable with the integral points of a $\Q$-form of $G$. The group $\Gamma(q)$ is the kernel of reduction modulo $q$, and is therefore normal in $\Gamma$.

      For a finite-index subgroup $\Gamma_i<\Gamma$, we write $X_i=\Gamma_i\bk\bH^n$. If $\Omega\subset X$ is the compact set in the collar property and $\Omega_i=p_i^{-1}(\Omega)$, then
      \begin{align*}
        \vol(\Omega_i)=D_i\vol(\Omega).
      \end{align*}
      We refer to $\Omega$ as the compact core. It is the compact set furnished by the collar property and need not be the convex core. If $x=\Gamma_i gK$, replacing $g$ by another representative conjugates the indexing group in \eqref{eq:energy}; consequently, the return energy is well-defined on $\Omega_i$.

      \section{The shadow-return estimate}\label{sec:transfer}

      We prove Theorem~\ref{thm:transfer}. This is a geometric formulation of the positive-kernel and collar argument used in \cite[Section~2.3]{H}. We retain the return terms rather than immediately estimating them by a global orbital count. Let
      \begin{align*}
        B_T=\{g\in G:d(o,go)\leq T\}.
      \end{align*}

      \begin{lemma}[Positive kernel]\label{lem:kernel}
        For every $T\geq1$ there is a bi-$K$-invariant function $F_T$ of positive type, supported in $\{g:d(o,go)\leq2T\}$, such that
        \begin{align}
          \widehat F_T(\lambda_u)&\gg_s e^{4(s-\rho)T}
          &&(u\in[s,\delta_\Gamma]),\label{eq:transform}\\
          0\leq F_T(g)&\ll_s e^{2(s-\rho)T}e^{-\rho d(o,go)}.
          \label{eq:kernel-bound}
        \end{align}
      \end{lemma}

      \begin{proof}
        Choose a smooth bi-$K$-invariant function $\chi_T$ which is equal to one on $B_{T-1}$, supported in $B_T$, and takes values in $[0,1]$. Let $f_T=\chi_T\psi_s$ and $F_T=f_T*\check f_T$. Then $\widehat F_T=|\widehat f_T|^2\geq0$. Uniformly for $u\in[s,\delta_\Gamma]$, \eqref{eq:spherical-decay} and the radial volume density give
        \begin{align*}
          \widehat f_T(\lambda_u)
          =\int_G\chi_T(g)\psi_s(g)\psi_u(g)\,\rd g
          \gg_s e^{(s+u-2\rho)T}
          \geq e^{2(s-\rho)T}.
        \end{align*}
        Here the bounded range $1\leq T\leq2$ is absorbed into the implied constant. This proves \eqref{eq:transform}. The support statement is immediate. The standard hyperbolic shell-intersection estimate gives \eqref{eq:kernel-bound}; see \cite{SX} and \cite[(2.3)]{H}.
      \end{proof}

      \begin{proof}[Proof of Theorem~\ref{thm:transfer}]
        Put $T=L/2$ and periodize $F_T$ on $\Gamma_i\bk G$:
        \begin{align*}
          K_{i,T}(g,h)=\sum_{\gamma\in\Gamma_i}F_T(g^{-1}\gamma h).
        \end{align*}
        Since $F_T$ is bi-$K$-invariant, this is a kernel on $X_i$. Let $R_i$ denote right convolution on $L^2(X_i)$, put $A_{i,T}=R_i(F_T)$, and let $M_i$ denote multiplication by $\mathbbm{1}_{\Omega_i}$. Since $F_T=f_T*\check f_T$, the operator $A_{i,T}$ is positive. Moreover, $M_iA_{i,T}M_i$ is trace class: indeed, $M_iR_i(f_T)$ is Hilbert--Schmidt, since $f_T$ is compactly supported and $\Omega_i$ is compact, and
        \begin{align*}
          M_iA_{i,T}M_i
          =(M_iR_i(f_T))(M_iR_i(f_T))^*.
        \end{align*}
        Cyclicity of the trace for products of Hilbert--Schmidt operators therefore gives
        \begin{align*}
          \int_{\Omega_i}K_{i,T}(g,g)\,\rd x
          &=\operatorname{Tr}(M_iA_{i,T}M_i)\\
          &=\operatorname{Tr}(A_{i,T}^{1/2}M_iA_{i,T}^{1/2}).
        \end{align*}
        If $\phi_j$ is an $L^2$-normalized eigenfunction with spectral parameter $u_j$, then
        \begin{align*}
          A_{i,T}\phi_j=\widehat F_T(\lambda_{u_j})\phi_j.
        \end{align*}
        Completing these eigenfunctions to an orthonormal basis of $L^2(X_i)$ and using the positivity of $A_{i,T}^{1/2}M_iA_{i,T}^{1/2}$, we obtain
        \begin{align*}
          \int_{\Omega_i}K_{i,T}(g,g)\,\rd x
          &\geq
          \sum_{u_j\in[s,\delta_\Gamma]}
          \widehat F_T(\lambda_{u_j})
          \int_{\Omega_i}|\phi_j|^2\,\rd x.
        \end{align*}
        Thus \eqref{eq:transform} and the collar property \eqref{CP} give
        \begin{align}\label{eq:spectral-side}
          \int_{\Omega_i}K_{i,T}(g,g)\,\rd x
          \gg_s e^{4(s-\rho)T}\cN(\lambda_s,\Gamma_i).
        \end{align}
        On the geometric side, \eqref{eq:kernel-bound} and the support of $F_T$ give
        \begin{align*}
          \frac1{D_i}\int_{\Omega_i}K_{i,T}(g,g)\,\rd x
          &\ll_s e^{2(s-\rho)T}\frac1{D_i}
          \int_{\Omega_i}
          \sum_{\substack{\gamma\in\Gamma_i\\r_g(\gamma)\leq2T}}
          e^{-\rho r_g(\gamma)}\,\rd x\\
          &=e^{2(s-\rho)T}\cE_i(2T).
        \end{align*}
        Comparing this with \eqref{eq:spectral-side} and using $L=2T$ proves \eqref{eq:transfer}.
      \end{proof}

      We finish this section by recalling the proof of the comparison estimates in Theorem~\ref{thm:SX-Oh}. This also isolates how the present geometric argument differs from the earlier counting arguments.

      \begin{proof}[Proof of Theorem~\ref{thm:SX-Oh}]
        For the first part, take $\Omega=X$, so the collar property is automatic. Since $X$ is compact, normality of $\Gamma(q)$ gives
        \begin{align}\label{eq:energy-count}
          \cE_q(L)
          \ll \sup_{g\in C}
          \sum_{\substack{\gamma\in\Gamma(q)\\d(go,\gamma go)\leq L}}
          e^{-\rho d(go,\gamma go)},
        \end{align}
        where $C\subset G$ is a fixed compact set projecting onto $X$. Sarnak and Xue prove, for the arithmetic lattices in the first part, the uniform counting estimate
        \begin{align}\label{eq:SX-count}
          \#\{\gamma\in\Gamma(q):d(go,\gamma go)\leq R\}
          \ll_\epsilon D_q^{-1}e^{(2\rho+\epsilon)R}+e^{\rho R}.
        \end{align}
        Moving the basepoint in $C$ changes $R$ by a bounded amount, so the estimate is uniform in $g$. Partial summation gives
        \begin{align*}
          \cE_q(L)
          \ll_\epsilon D_q^{-1}e^{(\rho+\epsilon)L}+e^{\epsilon L}.
        \end{align*}
        Theorem~\ref{thm:transfer} therefore yields
        \begin{align*}
          \frac{\cN(\lambda_s,\Gamma(q))}{D_q}
          \ll_{s,\epsilon}
          D_q^{-1}e^{(2\rho-s+\epsilon)L}
          +e^{-(s-\rho-\epsilon)L}.
        \end{align*}
        Taking $e^{\rho L}=D_q$ and renaming $\epsilon$ proves \eqref{eq:SX-density}. The arithmetic input is the lattice-point estimate \eqref{eq:SX-count}.

        For the second part, Oh's uniform orbital counting theorem \cite[Theorem~2.1]{H} gives, after decreasing $\eta$ if necessary so that $0<\eta<\delta_\Gamma-\rho$,
        \begin{align}\label{eq:Oh-count}
          \#\{\gamma\in\Gamma(q):d(go,\gamma go)\leq R\}
          \ll D_q^{-1}e^{\delta_\Gamma R}
          +e^{(\delta_\Gamma-\eta)R}
        \end{align}
        for prime $q$, uniformly on the fixed compact core. The same normality argument as in \eqref{eq:energy-count}, followed by partial summation, gives
        \begin{align*}
          \cE_q(L)
          \ll D_q^{-1}e^{(\delta_\Gamma-\rho)L}
          +e^{(\delta_\Gamma-\eta-\rho)L}.
        \end{align*}
        Using the collar property and Theorem~\ref{thm:transfer}, we obtain
        \begin{align*}
          \cN(\lambda_s,\Gamma(q))
          \ll_s e^{(\delta_\Gamma-s)L}
          +D_qe^{(\delta_\Gamma-\eta-s)L}.
        \end{align*}
        Taking $e^{\eta L}=D_q$ proves \eqref{eq:Oh-density}.
      \end{proof}

      \section{Principal congruence covers}\label{sec:congruence}

      We first show that a principal congruence cover has no nontrivial compact-core return below logarithmic scale. Combined with the pre-trace argument above, this immediately gives \eqref{eq:main}: this is precisely the argument of \cite[Section~2.3]{H}, stopped before the first nonidentity return. We then go beyond the injectivity scale in the convex-cocompact case and use shadow packing to prove \eqref{eq:main-cc}.

      \begin{lemma}[Compact-core congruence injectivity]\label{lem:injectivity}
        There is $B=B(\Gamma,\Omega)$ such that, for every $q\geq3$,
        \begin{align}\label{eq:injectivity}
          \inf_{\substack{x=\Gamma(q)gK\in\Omega_q\\e\neq\gamma\in\Gamma(q)}}
          d(go,\gamma go)\geq\log q-B.
        \end{align}
      \end{lemma}

      \begin{proof}
        Choose a compact set $C\subset G$ whose image in $X$ contains $\Omega$. Every point of $\Omega_q$ has a representative $\alpha gK$, with $\alpha\in\Gamma$ and $g\in C$. If $\eta\in\Gamma(q)$, then
        \begin{align*}
          d(\alpha go,\eta\alpha go)
          =d(go,\alpha^{-1}\eta\alpha go),
        \end{align*}
        and $\alpha^{-1}\eta\alpha\in\Gamma(q)$ by normality. It therefore suffices to prove a uniform estimate for $g\in C$.

        If $e\neq\gamma\in\Gamma(q)$, then some entry of $\gamma-I$ is a nonzero multiple of $q$. For any fixed matrix norm, $\|\gamma\|\gg q$. The Cartan decomposition in the defining representation gives
        \begin{align*}
          \|\gamma\|\ll_G e^{d(o,\gamma o)},
        \end{align*}
        and therefore $d(o,\gamma o)\geq\log q-O_G(1)$. For $g\in C$,
        \begin{align*}
          d(go,\gamma go)
          \geq d(o,\gamma o)-2d(o,go)
          \geq\log q-B.
        \end{align*}
        The preceding reduction gives the same estimate on every sheet of $\Omega_q$.
      \end{proof}

      We next use the separation in Lemma~\ref{lem:injectivity} to include nonidentity returns. In the convex-cocompact case, Ahlfors regularity of the Patterson--Sullivan measure controls how many such returns can occur in a radial shell.

      \begin{lemma}[Packing of enlarged shadows]\label{lem:shadow-packing}
        Suppose that $\Gamma$ is convex cocompact, and put
        \begin{align*}
          \ell_q=\log q-B,
        \end{align*}
        where $B$ is as in Lemma~\ref{lem:injectivity}. For every $0<\vartheta<1/2$, uniformly for $x=\Gamma(q)gK\in\Omega_q$ and $R\geq\ell_q$,
        \begin{align}\label{eq:shell-packing}
          \#\{\gamma\in\Gamma(q):R\leq r_g(\gamma)<R+1\}
          \ll_\vartheta
          e^{\delta_\Gamma(R-\vartheta\ell_q)}.
        \end{align}
        Consequently, for $L\geq\ell_q$,
        \begin{align}\label{eq:energy-packing}
          \cE_q(L)
          \ll_\vartheta
          1+e^{-\vartheta\delta_\Gamma\ell_q}
          e^{(\delta_\Gamma-\rho)L}.
        \end{align}
      \end{lemma}

      \begin{proof}
        Put $z=go$. The points $\gamma z$, $\gamma\in\Gamma(q)$, are pairwise $\ell_q$-separated: if $\gamma_1\neq\gamma_2$, then Lemma~\ref{lem:injectivity} applied to $\gamma_1^{-1}\gamma_2$ gives
        \begin{align*}
          d(\gamma_1z,\gamma_2z)
          =d(z,\gamma_1^{-1}\gamma_2z)\geq\ell_q.
        \end{align*}
        Consider those orbit points for which $R\leq d(z,\gamma z)<R+1$. Since $\Gamma$ is convex cocompact, $\nu_z$ is Ahlfors $\delta_\Gamma$-regular for a visual metric based at $z$. Standard shadow geometry therefore gives
        \begin{align}\label{eq:enlarged-shadow}
          \nu_z\!\left(
          \operatorname{Sh}_z(B_{\gamma z}(\vartheta\ell_q))
          \right)
          \asymp_\vartheta
          e^{-\delta_\Gamma(R-\vartheta\ell_q)}.
        \end{align}
        The constants are uniform because the image of $z$ lies in the fixed compact set $\Omega$.

        These enlarged shadows have uniformly bounded overlap. Indeed, if a boundary point belongs to two of them, the Gromov product of the corresponding centers, based at $z$, is at least $R-\vartheta\ell_q-O(1)$. Hence the two centers are at distance at most
        \begin{align*}
          2\vartheta\ell_q+O(1)<\ell_q
        \end{align*}
        for all sufficiently large $q$, contradicting their $\ell_q$-separation. The finitely many remaining levels are absorbed into the implied constant. Summing \eqref{eq:enlarged-shadow} and using the finiteness of $\nu_z$ proves \eqref{eq:shell-packing}.

        Lemma~\ref{lem:injectivity} shows that the identity is the only return with $r_g(\gamma)<\ell_q$. Decomposing the remaining returns into unit shells and applying \eqref{eq:shell-packing}, we obtain
        \begin{align*}
          \sum_{\substack{\gamma\in\Gamma(q)\\r_g(\gamma)\leq L}}
          e^{-\rho r_g(\gamma)}
          &\ll_\vartheta
          1+e^{-\vartheta\delta_\Gamma\ell_q}
          \sum_{\ell_q\leq R\leq L}
          e^{(\delta_\Gamma-\rho)R}\\
          &\ll_\vartheta
          1+e^{-\vartheta\delta_\Gamma\ell_q}
          e^{(\delta_\Gamma-\rho)L}.
        \end{align*}
        Integrating over $\Omega_q$ and dividing by $D_q$ proves \eqref{eq:energy-packing}.
      \end{proof}

      \begin{proof}[Proof of Theorem~\ref{thm:main}]
        Fix $s\in(\rho,\delta_\Gamma]$ and take the compact set $\Omega=\Omega_s$ furnished by the collar lemma. For all sufficiently large $q$, so that the quantity below is at least $2$, put
        \begin{align*}
          L=\log q-B-1.
        \end{align*}
        By Lemma~\ref{lem:injectivity}, the identity is the only term in \eqref{eq:energy}, and hence
        \begin{align*}
          \cE_q(L)=\vol(\Omega).
        \end{align*}
        Equivalently, in the unfolded expression from \cite[Section~2.3]{H}, every nonidentity term vanishes by the support of $F_T$, and the identity contributes $D_q\vol(\Omega)F_T(e)$. Since $F_T(e)\ll_s e^{2(s-\rho)T}$, comparison with the spectral lower bound \eqref{eq:spectral-side} gives the same estimate directly.
        Theorem~\ref{thm:transfer} gives
        \begin{align*}
          \frac{\cN(\lambda_s,\Gamma(q))}{D_q}
          \ll_s e^{-(s-\rho)(\log q-B-1)}
          \ll_s q^{-(s-\rho)}.
        \end{align*}
        This proves \eqref{eq:main}; the remaining finitely many levels are absorbed into the implied constant.

        Suppose now that $\Gamma$ is convex cocompact. If $\delta_\Gamma=2\rho$, then \eqref{eq:main-cc} is the same as \eqref{eq:main}, after enlarging the implied constant. We may therefore assume that $\delta_\Gamma<2\rho$. Put
        \begin{align*}
          d=\delta_\Gamma-\rho,
          \qquad x=s-\rho.
        \end{align*}
        Choose
        \begin{align*}
          \frac{d}{\delta_\Gamma}<\vartheta<\frac12
        \end{align*}
        and set
        \begin{align*}
          L=\frac{\vartheta\delta_\Gamma}{d}\ell_q.
        \end{align*}
        Thus $L\geq\ell_q$ for all sufficiently large $q$. By Theorem~\ref{thm:transfer} and \eqref{eq:energy-packing},
        \begin{align*}
          \frac{\cN(\lambda_s,\Gamma(q))}{D_q}
          &\ll_{s,\vartheta}
          e^{-xL}
          \left(1+e^{-\vartheta\delta_\Gamma\ell_q}e^{dL}\right)\\
          &=2e^{-xL}\\
          &=2\exp\!\left(
          -\frac{\vartheta\delta_\Gamma(s-\rho)}{\delta_\Gamma-\rho}\ell_q
          \right).
        \end{align*}
        Since $\ell_q=\log q-O(1)$, taking $\vartheta<1/2$ sufficiently close to $1/2$ proves \eqref{eq:main-cc}.
      \end{proof}

      \section*{Acknowledgements}
      The author used OpenAI's ChatGPT for research brainstorming, literature discovery, checking intermediate arguments and calculations, and editorial revision.  ChatGPT was not treated as an author or as an independent source of mathematical authority.  The author assumes full responsibility for all statements, proofs, citations, and errors.

      This work was supported by a grant from the Simons Foundation [SFI-MPS-TSM-00013410,CL]. 

      We thank Dennis Sullivan for posing the question that motivated this work and Jonathan Fraser for a helpful conversation. We thank Hee Oh helpful comments on an earlier draft.

  \small 
  \bibliographystyle{alpha}
  \bibliography{biblio}

    \hrulefill

    \vspace{4mm}
    \noindent Mathematics Department, University of Houston, Houston, TX, 77004, USA
    \\
    \emph{E-mail: \textbf{clutsko@uh.edu}}

\end{document}